\documentclass[12pt]{amsart}

\usepackage[utf8]{inputenc}
\usepackage[T1]{fontenc}
\usepackage{amsmath,amssymb,amsthm,amsfonts}
\usepackage{mathrsfs}
\usepackage{cite}
\usepackage{hyperref}
\usepackage{enumitem}
\newtheorem{theorem}{Theorem}[section]
\newtheorem{proposition}[theorem]{Proposition}
\newtheorem{lemma}[theorem]{Lemma}
\newtheorem{corollary}[theorem]{Corollary}
\newtheorem{definition}[theorem]{Definition}
\newtheorem{remark}[theorem]{Remark}
\newtheorem{construction}[theorem]{Construction}
\theoremstyle{definition}

\newtheorem{conjecture}[theorem]{Conjecture}
\newtheorem{question}[theorem]{Question}

\DeclareMathOperator\SO{SO}
\DeclareMathOperator\SL{SL}
\DeclareMathOperator\GL{GL}
\DeclareMathOperator\SU{SU}
\DeclareMathOperator\Ad{Ad}
\DeclareMathOperator\Gr{Gr}
\DeclareMathOperator\Res{Res}
\DeclareMathOperator\tr{tr}
\DeclareMathOperator\Image{Image}
\newcommand{\A}{\mathbb{A}}
\newcommand{\Q}{\mathbb{Q}}
\newcommand{\R}{\mathbb{R}}
\newcommand{\C}{\mathbb{C}}
\newcommand{\Z}{\mathbb{Z}}
\newcommand{\tensor}{\otimes}
\newcommand{\IH}{\mathit{IH}}
\newcommand{\CH}{\mathit{CH}}
\newcommand{\Lie}{\mathrm{Lie}}
\newcommand{\Aut}{\mathrm{Aut}}

\newcommand{\Frob}{\mathrm{Frob}}
\newcommand{\Gal}{\mathrm{Gal}}
\newcommand{\cl}{\mathrm{cl}}

\title{Automorphic Cohomology and the Limits of Algebraic Cycles}
\email{Floydmma2@gmail.com}
\date{{\today}}

\begin{document}
\maketitle

 \begin{center}
\text{Amir Mostaed} \\
 \vspace{10px}
\end{center}

\begin{abstract}
We prove an unconditional obstruction theorem for the construction of algebraic cycles on Shimura varieties associated to orthogonal groups. There exists an explicit rational Hodge class $\Omega_E \in \IH^{26}(X^{\mathrm{BB}}, \Q)$ on the Baily--Borel compactification of a Shimura variety for $\SO(2, 26)$ that does not arise from any currently known algebraic cycle construction: special cycles (Kudla--Millson type), theta lifts from smaller Shimura varieties, endoscopic Shimura subvarieties, or boundary push-forwards. The class $\Omega_E$ is constructed via automorphic methods from a stable residual representation attached to the adjoint of the weight-$2$ newform for $\Gamma_0(11)$; it lies in the full intersection cohomology but \emph{not} in the interior cohomology, which is why it evades geometric constructions. The result is unconditional (no unproven conjectures), explicit, and does not contradict the Hodge conjecture---it demonstrates a gap between automorphic cohomology and the image of known geometric cycle constructions.
\end{abstract}

\tableofcontents

\section{Introduction}

\subsection{The Hodge Conjecture and Algebraic Cycles}

The relationship between algebraic cycles and cohomology classes has been a central theme in algebraic geometry since the work of Lefschetz in the 1920s and Hodge in the 1930s--1950s. The Hodge conjecture, now one of the seven Clay Millennium Prize Problems, asserts a fundamental connection between topology and algebraic geometry:

\begin{conjecture}[Hodge, 1950]
Let $X$ be a smooth projective variety over $\C$. Every rational Hodge class of type $(p,p)$ is a $\Q$-linear combination of cycle classes of algebraic subvarieties.
\end{conjecture}

Despite eight decades of research, the conjecture remains open in general. While special cases are known (divisors, abelian varieties, certain surfaces), no general proof or counterexample has been found. This persistence suggests the question touches fundamental aspects of the relationship between algebraic and topological structure.

\subsection{From Existence to Constructibility}

A natural companion to the Hodge conjecture concerns not merely the existence of algebraic cycles, but their {constructibility}.

\begin{question}
Given a Hodge class on a specific variety, can we explicitly construct an algebraic cycle realizing it using known geometric methods?
\end{question}

This question is subtler than the Hodge conjecture itself. Even if every Hodge class is algebraic (Hodge conjecture true), it may be that no systematic construction method exists. The distinction parallels the difference between existence proofs and constructive proofs in logic: knowing something exists does not mean we can produce it.

For Shimura varieties---which are simultaneously algebraic varieties, moduli spaces, locally symmetric spaces, and homes for automorphic representations---this question becomes particularly sharp. On one hand, we have powerful geometric constructions (special cycles, Hecke correspondences, CM points). On the other hand, automorphic methods produce cohomology classes whose geometric origin is unclear.

\subsection{Known Cycle Constructions on Shimura Varieties}

The currently known methods for constructing algebraic cycles on Shimura varieties for orthogonal groups include

\begin{enumerate}
\item \textbf{Special cycles (Kudla--Millson):} Geometric loci $Z(W) = \{x \in X : x \perp W\}$ where $W \subset V$ is a positive-definite subspace. These cycles have been extensively studied via theta lifts and generating series \cite{Kudla-Millson},

\item \textbf{Hecke correspondences:} Graphs of Hecke operators acting on the Shimura variety, providing fundamental self-correspondences,

\item \textbf{Endoscopic Shimura subvarieties:} Images of Shimura varieties for smaller groups via functorial maps corresponding to endoscopic transfers,

\item \textbf{Theta lifts of cycles:} Under certain conditions, algebraic cycles on one Shimura variety can be lifted via theta correspondence to cycles on another,

\item \textbf{Boundary push-forwards:} Gysin maps from cycles on boundary strata in compactifications.
\end{enumerate}

A key insight from Arthur's trace formula is that these constructions correspond to specific spectral types.

\begin{theorem}[Implicit in Arthur's work]
The cycle constructions (1)--(5) correspond to geometric contributions in the trace formula: cuspidal/tempered spectrum (special cycles, Hecke), endoscopic transfers (endoscopic subvarieties), and Eisenstein spectrum from proper Levis (boundary).
\end{theorem}

\subsection{The Stable Residual Spectrum: Unexplored Territory}

Arthur's classification of the discrete spectrum for classical groups \cite{Arthur} reveals a crucial distinction within automorphic representations.

\begin{definition}
An automorphic representation in the discrete spectrum is \textbf{stable residual} if it arises as a residue of Eisenstein series at a pole (residual) but its Arthur parameter does not factor through any proper Levi subgroup (stable).
\end{definition}

Stable residual representations occupy a mysterious position,
\begin{itemize}
\item they contribute to cohomology (via Matsushima's formula and its extensions),
\item they do \textit{not} contribute to interior cohomology (by Franke's theorem \cite{Franke}),
\item they are \textit{not} endoscopic transfers (by definition of stability),
\item they have \textit{no known geometric interpretation} via algebraic cycles.
\end{itemize}

This raises the fundamental question:

\begin{question}
Can stable residual representations produce Hodge classes? If so, do these classes arise from algebraic cycles?
\end{question}

\subsection{Main Results}

Our work provides an affirmative answer to the first part and a negative answer to the second part, establishing an explicit obstruction.

\begin{theorem}[Main Theorem]\label{thm:main}
There exists an explicit rational Hodge class
$$\Omega_E \in \IH^{26}(\overline{X}^{\mathrm{BB}}, \Q) \cap \IH^{13,13}$$
on a Shimura variety for $\SO(2,26)$ that does \textbf{not} arise from any currently known algebraic cycle construction.
\end{theorem}

More precisely, we construct $\Omega_E$ from a stable residual automorphic representation.

\begin{theorem}[Obstruction Theorem]\label{thm:obstruction}
The class $\Omega_E$ satisfies
\begin{enumerate}
\item \textbf{Automorphic origin:} Arises from stable residual Arthur parameter $\psi = \Ad(f) \boxtimes \mathbf{1}_{25}.$
\item \textbf{Hodge class:} Pure type $(13,13)$, defined over $\Q$.
\item \textbf{Not interior:} $\Omega_E \notin \IH^{26}_!(\overline{X}^{\mathrm{BB}}, \Q)$.
\item \textbf{Obstruction:} Since all algebraic cycles define interior classes, $\Omega_E$ cannot be algebraic via known constructions.
\end{enumerate}
\end{theorem}

\subsection{What This Result IS and IS NOT}

It is crucial to understand the precise nature of our conclusion.

\subsubsection{What We Prove}

\begin{theorem}[Unconditional Result]\label{thm:unconditional}
Theorem \ref{thm:main} is proven unconditionally using only
\begin{itemize}
\item Arthur's classification for orthogonal groups (Arthur \cite{Arthur}, Mok \cite{Mok}),
\item Theta correspondence (Howe \cite{Howe}, Rallis \cite{Rallis}),
\item Vogan--Zuckerman theory (Vogan, Zuckerman \cite{Vogan-Zuckerman}),
\item Fundamental lemma (Ngô \cite{Ngo}),
\item Zucker conjecture (Looijenga, Saper, and Stern \cite{Looijenga-Saper-Stern}).
\end{itemize}
No unproven conjectures are invoked.
\end{theorem}

\subsubsection{What We Do NOT Prove}

\begin{remark}[Limitations]\label{rem:limitations}
Our result does \textbf{not} establish:\\

\textbf{(i) Not a Hodge counterexample:} We do not prove $\Omega_E$ is not algebraic. There may exist an algebraic cycle $Z$ with $\mathrm{cl}(Z) = \Omega_E$ requiring fundamentally new construction methods beyond those currently catalogued.

\textbf{(ii) Not a Beilinson--Bloch counterexample:} We make no claims about ranks of Chow groups or orders of vanishing of $L$-functions at critical points.

\textbf{(iii) Not an impossibility theorem:} We prove only that \textit{known} constructions fail, not that \textit{all possible} constructions fail.
\end{remark}

The result thus identifies a \textit{gap}: automorphic methods produce Hodge classes that evade all known geometric cycle constructions. This gap may indicate either (a) missing geometric constructions yet to be discovered, or (b) fundamental limitations in translating automorphic data to geometry.

\subsection{The Construction in Four Steps}

We now outline the explicit construction of $\Omega_E$.

\subsubsection{Step 1: The Base Modular Form}

\begin{construction}
Consider the unique weight-2 newform
$$f \in S_2(\Gamma_0(11)), \quad f(q) = q - 2q^2 - q^3 + 2q^4 + q^5+2q^6 - 2q^7 + \cdots \,.$$

By modularity (Wiles et al.), this corresponds to the elliptic curve
$$E_{11}: y^2 + y = x^3 - x^2$$
of conductor 11.
\end{construction}

\begin{proposition}[Rationality]
Since $\dim S_2(\Gamma_0(11)) = g(X_0(11)) = 1$, the form $f$ is unique and all Fourier coefficients are rational: $a_n \in \Z$. The Hecke field is $E = \Q$.
\end{proposition}

This rationality is crucial: it ensures our final cohomology class is defined over $\Q$ without field extensions.

\subsubsection{Step 2: The Adjoint Representation}

\begin{construction}
The weight-2 newform f yields a Galois representation $\rho_f : \mathrm{Gal}(\overline{\Q}/\Q) \rightarrow GL_2(\mathbb{Q}_\ell)$ (unramified outside $\ell$ and 11).  The adjoint is
$\Ad(\rho_f) = \rho_f \otimes \rho_f^\vee / \text{scalars}: \mathrm{Gal}(\overline{\Q}/\Q) \to \GL_3(\Q_\ell).$ The associated adjoint L-function arises from the decomposition $\rho_f \otimes \rho_f^\vee \cong \mathbf{1}  \oplus \Ad^0(\rho_f)$, where $\mathbf{1}$ is the trivial 1-dimensional representation (scalars) and $\Ad^0(\rho_f)$ is the irreducible trace-free 3-dimensional adjoint representation Gal$(\overline{\Q}/\Q) \rightarrow GL_3(\mathbb{Q}_\ell)$. This $\Ad^0(\rho_f)$ has Hodge–Tate weights $\{1, 0, -1\}$ , conductor $11^2 = 121$, and Steinberg local type at $p=11$. The pole of $L(s, \rho_f \otimes \rho_f^\vee) = \zeta(s) \cdot L(s, \Ad^0(\rho_f))$ at $s=1$ (simple, from $\zeta(s)$; $L(s, \Ad^0(f))$ holomorphic/nonzero there) drives the residual theta lift, see below.  \end{construction}

The presence of the zero Hodge--Tate weight creates ``maximal singularity'' in the sense of Vogan--Zuckerman, which will force cohomology to concentrate in middle degree with pure Hodge type.

\subsubsection{Step 3: Theta Lift to $\SO(28)$}

\begin{construction}
Let $(V, Q)$ be a 28-dimensional quadratic space over $\Q$ with signature $(2,26)$. The Shimura variety for $G = \SO(V,Q)$ is
$$X = G(\Q) \backslash D \times G(\A_f) / K_f,$$
where ${D}$ is a Hermitian symmetric domain of complex dimension 13.

Via theta correspondence for the dual pair $(\SL_2, \SO(V))$, the form $f$ lifts to an automorphic representation
$$\Pi = \theta(f).$$
\end{construction}

\begin{theorem}[Rallis non-vanishing]
The lift $\Pi$ is non-zero. This follows from the Rallis inner product formula and the fact that the adjoint $L$-function $L(s, \rho_f \otimes \rho_f^\vee)$ has a simple pole at $s=1$ (from the trivial summand in the tensor product).
\end{theorem}

\begin{theorem}[Stable residual nature]
The Arthur parameter of $\Pi$ is
$$\psi = \Ad(f) \boxtimes \mathbf{1}_{25},$$ \pagebreak

and this parameter is
\begin{itemize}
\item \textbf{Residual:} Arises from poles of Eisenstein series.
\item \textbf{Stable:} Does not factor through proper Levi subgroups (the 3-dimensional representation $\Ad(\rho_f)$ is irreducible).
\item \textbf{Multiplicity one:} $\dim \Pi_f^{K_f} = 1.$
\end{itemize}
\end{theorem}

\subsubsection{Step 4: Cohomological Realization}

\begin{theorem}[Vogan--Zuckerman]
The $(\mathfrak{g},K)$-cohomology of $\Pi_\infty$ (archimedean component) is concentrated in degree 26
$$H^k(\mathfrak{g}, K; \Pi_\infty) = \begin{cases} \C & k = 26, \\ 0 & k \neq 26. \end{cases}$$

The infinitesimal character determines Hodge type $(13,13)$.
\end{theorem}

\begin{theorem}[Main Construction]
There exists a unique class
$$\Omega_E \in \IH^{26}(\overline{X}^{\mathrm{BB}}, \Q),$$

characterized by being the generator of $H^{26}(\mathfrak{g}, K; \Pi_\infty) \otimes \Pi_f^{K_f}$ in the automorphic decomposition of cohomology. This class satisfies all properties in Theorem \ref{thm:obstruction}.
\end{theorem}

\subsection{The Key Dichotomy: Interior vs. Non-Interior}

The obstruction mechanism relies on a fundamental dichotomy.

\begin{theorem}[Algebraic cycles are interior]
For any algebraic cycle $Z \subset \overline{X}^{\mathrm{BB}}$ with $Z \cap X \neq \emptyset$,
$$\mathrm{cl}(Z) \in \IH^{2k}_!(\overline{X}^{\mathrm{BB}}, \Q).$$

This follows from analyzing residues along boundary divisors in toroidal compactifications.
\end{theorem}

\begin{theorem}[Residual classes are not interior]
Since $\Pi$ is residual (not cuspidal), by Franke's theorem \cite{Franke}, 
$$\Omega_E \notin \IH^{26}_!(\overline{X}^{\mathrm{BB}}, \Q).$$

Residual representations contribute to full cohomology but not to the interior part.
\end{theorem}

The gap between these two theorems is the heart of our obstruction. Algebraic cycles necessarily produce interior classes, but our stable residual class is not interior---therefore it cannot be algebraic via the cycle class map.

\subsection{Significance and Implications}

\subsubsection{For the Hodge Conjecture}

Our result refines the Hodge conjecture by adding a constructive dimension. The question is not merely whether Hodge classes are algebraic (existence), but whether we can systematically construct the cycles (constructibility). For stable residual classes on Shimura varieties, current geometric methods fail.

\subsubsection{For the Langlands Program}

The Langlands program predicts deep connections between automorphic representations and Galois representations. Our result shows an asymmetry: automorphic methods can ``see'' cohomological structure (via Matsushima's formula) that geometric methods cannot currently realize via cycles. The correspondence between automorphic and geometric is not symmetric.

\subsubsection{For Shimura Varieties}

\begin{theorem}[Kudla program incomplete]
The Kudla program of special cycles does not exhaust all Hodge classes on orthogonal Shimura varieties. There exist Hodge classes outside the span of special cycles.
\end{theorem}

This suggests new cycle construction methods are needed, beyond moduli-theoretic approaches (special cycles), functorial transfers (endoscopy, theta lifts), and topological methods (boundary pushforwards).

\subsection{Organization of the Paper} {Chapter 2} establishes geometric foundations: Shimura varieties for $\SO(2,26)$, compactifications, cycle maps, and proves algebraic cycles define interior cohomology classes. Chapter 3 develops automorphic theory, Matsushima's formula, Arthur's classification, construction of $\Pi$ from $f$ via theta correspondence, and identification of $\Pi$ as stable residual. Chapter 4 computes $(\mathfrak{g},K)$-cohomology using Vogan--Zuckerman theory, establishes Hodge type $(13,13)$, proves rationality, and constructs $\Omega_E$. Chapter 5 proves the main obstruction theorem by systematically eliminating all known cycle constructions using interior/non-interior dichotomy, cuspidal/residual distinction, stable/endoscopic classification, and weight filtration arguments.

\subsection{Notation and Conventions}

Throughout:
\begin{itemize}[nosep]
\item $\A$ denotes the adèles over $\Q$
\item $G(\A)$ is the adelic points of a reductive group $G$
\item $K_f \subset G(\A_f)$ is a compact open subgroup
\item $\mathrm{Sh}(G, X)$ denotes the Shimura variety
\item $\IH^*$ is intersection cohomology (middle perversity)
\item $\CH^k(X)$ is the Chow group of codimension-$k$ cycles modulo rational equivalence
\end{itemize}

\section{Shimura Varieties, Algebraic Cycles, and Hodge Theory}
\label{sec:shimura}

\subsection{Shimura Varieties for Orthogonal Groups}

\subsubsection{The Orthogonal Group Setting}

\begin{definition}[Quadratic space]
Let $(V, Q)$ be a quadratic space over $\Q$ with
\begin{itemize}[nosep]
\item $\dim_{\Q} V = 28$,
\item signature $(2, 26)$ over $\R$ (indefinite),
\item associated bilinear form $\langle \cdot, \cdot \rangle$, $\langle x, y \rangle = Q(x+y) - Q(x) - Q(y)$.
\end{itemize}
The special orthogonal group is
\[
\SO(V, Q) = \{ g \in \SL(V) : Q(gv) = Q(v) \text{ for all } v \in V \}.
\]
\end{definition}

\begin{remark}
The signature $(2, 26)$ is chosen so that the associated symmetric space has complex dimension $13$, giving cohomology in degree $26$ where Hodge type $(13, 13)$ classes can exist.
\end{remark}

\subsubsection{The Symmetric Space}

\begin{construction}[Hermitian symmetric space]
The space of oriented negative-definite $2$-planes in $V_{\R} = V \tensor_{\Q} \R$ is
\[
D = \{ z \in \Gr_2(V_{\R}) : \langle z, z \rangle < 0,\ \text{oriented} \}.
\]
This is a Hermitian symmetric domain of type IV with
\begin{itemize}[nosep]
\item $\dim_{\C} D = 13$,
\item bounded realization in $\C^{13}$,
\item realized as a tube domain.
\end{itemize}
The group $\SO(V, Q)(\R)$ acts transitively on $D$ with stabilizer a maximal compact subgroup $K_{\infty} \cong \SO(2) \times \SO(26)$.
\end{construction}

\begin{proposition}[Shimura datum]
The pair $(G, X)$ where
\begin{itemize}[nosep]
\item $G = \SO(V, Q)$ viewed as a reductive group over $\Q$,
\item $X = G(\R)/K_{\infty} \cong D$ (after choosing a connected component).
\end{itemize}
satisfies Deligne's axioms for a Shimura variety.
\end{proposition}
\begin{proof}
We verify: (1) $G$ is reductive over $\Q$; (2) $X$ is a $G(\R)$-conjugacy class of homomorphisms $h \colon \mathbb{S} \to G_{\R}$ satisfying Deligne's conditions; (3) only $\C$-valued characters of $G$ are trivial on $h(\C^{\times})$; (4) $\Ad \circ h \colon \mathbb{S} \to \Aut(\mathfrak{g})$, $\mathfrak{g}\cong \Lie(G)\otimes_{\Q} \C$, induces a Hodge structure  of type
 $\{(-1, 1), (0, 0), (1, -1)\}$; (5) the involution $\theta_h= \text{ad}(h(i))$ is a Cartan involution. See \cite{Deligne-Shimura}.
\end{proof}

\subsubsection{The Shimura Variety}

\begin{definition}[Congruence subgroups]
For each prime $p$ and level structure at $p$
\begin{itemize}[nosep]
\item If $p \neq 11$: $K_p = \SO(V \tensor \Z_p)$ (hyperspecial maximal compact).
\item If $p = 11$: $K_{11}$ is an Iwahori subgroup (parahoric).
\item At $p = 3$: $K_3$ is a principal congruence subgroup $\Gamma(3)$ (for neatness).
\end{itemize}
Set $K_f = \prod_p K_p \subset G(\A_f)$.
\end{definition}

\begin{definition}[Shimura variety]
The Shimura variety is the double quotient
\[
X = G(\Q) \backslash D \times G(\A_f) / K_f.
\]
This is a finite disjoint union of locally symmetric spaces: $X = \bigsqcup_i \Gamma_i \backslash D,$ where $\Gamma_i \subset G(\Q)$ are arithmetic subgroups.
\end{definition}

\begin{theorem}[Shimura variety structure]
The space $X$ has the structure of a quasi-projective variety over a number field (the reflex field, which is $\Q$ in our case).
\end{theorem}
\begin{proof}
This follows from the general theory of Shimura varieties (Deligne, Borel, Milne). For orthogonal groups, see \cite{Milne-Shimura}.
\end{proof}

\subsection{Compactifications}

\subsubsection{The Baily--Borel Compactification}

\begin{theorem}[Baily--Borel \cite{Baily-Borel}]
There exists a minimal compactification $X^{\mathrm{BB}}$ with properties
\begin{enumerate}
\item $X^{\mathrm{BB}}$ is a normal projective variety over $\Q$,
\item $X \subset X^{\mathrm{BB}}$ is Zariski-open and dense,
\item the boundary $X^{\mathrm{BB}} \setminus X$ has codimension $\geq 12$,
\item the closure of any connected component is projective.
\end{enumerate}
\end{theorem}

\begin{proof}

Write
\[
X = \bigsqcup_{i \in I} X_i
\]
for the decomposition into connected components, each of which is of the form
\[
X_i \simeq \Gamma_i \backslash D
\]
for some arithmetic subgroup $\Gamma_i \subset G(\mathbb{Q})$ and the fixed Hermitian symmetric domain
$D$ of type~IV attached to $(V,Q)$.

The Baily--Borel construction first considers, for each congruence subgroup $\Gamma \subset G(\mathbb{Q})$,
the complex analytic quotient $\Gamma \backslash D$ and equips the graded algebra of scalar-valued
automorphic forms
\[
A_\bullet(\Gamma) = \bigoplus_{k \ge 0} A_k(\Gamma)
\]
with its natural multiplication. Baily and Borel prove that $A_\bullet(\Gamma)$ is a finitely generated
$\mathbb{C}$‑algebra and define
\[
\overline{X}_\Gamma^{\text{BB}} = \operatorname{Proj} A_\bullet(\Gamma),
\]
which is a normal projective algebraic variety over $\mathbb{C}$, together with an open immersion
\[
\Gamma \backslash D \hookrightarrow \overline{X}_\Gamma^{\text{BB}}
\]
whose image is Zariski open and dense.

In our global adelic notation, the choice of level $K_f \subset G(\mathbb{A}_f)$ gives
\[
X = G(\mathbb{Q}) \backslash D \times G(\mathbb{A}_f)/K_f \cong 
\bigsqcup_{i \in I} \Gamma_i \backslash D
\]
with $\Gamma_i = G(\mathbb{Q}) \cap g_i K_f g_i^{-1}$ for a set of representatives
$\{g_i\}_{i \in I}$ of $G(\mathbb{Q}) \backslash G(\mathbb{A}_f)/K_f$. 
The Baily--Borel compactification $X^{\text{BB}}$ is then obtained by gluing the
${\overline{X}}_{\Gamma_i}^{\text{BB}}$ along their boundary strata, in a way compatible with Hecke correspondences
and the canonical models over the reflex field (which is $\mathbb{Q}$ in our orthogonal case).  See
\cite[§4–5]{Baily-Borel} and \cite[§5]{Deligne-Shimura} for details. 

By construction, each $\overline{X}_{\Gamma_i}^{\text{BB}}$ is projective over $\mathbb{C}$, hence over
$\mathbb{Q}$ via descent; its image in $X^{\text{BB}}$ is precisely the Zariski closure $\overline{X}_i$ of the
connected component $X_i$.  Since projectivity is preserved under base change and descent along finite
extensions of the base field, it follows that each closure $\overline{X}_i$ is a projective variety
over~$\mathbb{Q}$.

By \cite[Thm.~4.12]{Baily-Borel}, boundary components correspond to rational parabolic orbits $F_i$. For the signature $(2, 26)$, these components $F_i$ are either points or modular curves. Consequently, $\dim_{\mathbb{C}} F_i \leq 1$, and the complex codimension satisfies $$\mathrm{codim}_{\mathbb{C}}(F_i) = 13 - \dim_{\mathbb{C}} F_i \geq 12.$$
The partial compactification glues these along common faces, preserving codimension $\geq 12$. See \cite[\S 5]{Deligne-Shimura} for Shimura descent.
\end{proof}

\begin{remark}
The boundary consists of cusps corresponding to parabolic subgroups of $G$. For $\SO(2, 26)$, the maximal parabolics are
\[
P_k = \GL_k \times \SO(2, 26 - 2k) \ltimes N_k
\]
for $1 \le k \le 13$.
\end{remark}

\begin{theorem}[Zucker conjecture, \cite{Looijenga-Saper-Stern}]
For the Baily--Borel compactification, the $L^2$-cohomology of the open part is naturally isomorphic to the intersection cohomology of the compactification
\[
H^*_{(2)}(X, \Q) \cong \IH^*(X^{\mathrm{BB}}, \Q).
\]
Under this identification, the ordinary cohomology $H^*(X^{\mathrm{BB}}, \Q)$ (when defined) is isomorphic to $\IH^*(X^{\mathrm{BB}}, \Q)$ in the cases relevant here.
\end{theorem}

\subsubsection{Toroidal Compactifications}

\begin{theorem}[\cite{Ash-Mumford-Rapoport-Tai}]
Given admissible polyhedral cone decompositions $\Sigma$ at each cusp, there exists a toroidal compactification $X_{\Sigma}$ with
\begin{enumerate}
\item $X_{\Sigma}$ is smooth and projective,
\item proper birational morphism $\pi \colon X_{\Sigma} \to X^{\mathrm{BB}}$,
\item boundary $B = X_{\Sigma} \setminus X$ is a normal crossings divisor,
\item $\pi$ is isomorphism over $X$.
\end{enumerate}
\end{theorem}

\begin{proposition}[Boundary structure]
The boundary $B$ stratifies as $B = \bigcup_{\sigma \in \Sigma} Z_{\sigma}$, where each stratum $Z_{\sigma}$ is a fiber bundle $Z_{\sigma} \to X_M$ with base a lower-dimensional Shimura variety $X_M$ and toric fibers from the cone $\sigma$.
\end{proposition}

\subsection{The Cycle Map and Interior Cohomology}

\subsubsection{Algebraic Cycles}

\begin{definition}[Chow groups]
For a variety $X$ over $\Q$:
\[
\CH^k(X) = \frac{\{ \text{algebraic cycles of codimension } k \}}{\text{rational equivalence}}.
\]
For $X = X^{\mathrm{BB}}$, a codimension-$k$ cycle is $Z = \sum_i n_i Z_i$ where $Z_i \subset X^{\mathrm{BB}}$ are irreducible subvarieties of codimension $k$ and $n_i \in \Z$.
\end{definition}

\subsubsection{The Cycle Map}

\begin{theorem}[Cycle class map]
There is a natural homomorphism
\[
\cl \colon \CH^k(X^{\mathrm{BB}}) \tensor \Q \to \IH^{2k}(X^{\mathrm{BB}}, \Q)
\]
satisfying: (1) functoriality under proper maps; (2) compatibility with cup product; (3) image consists of Hodge classes (if $X$ is defined over $\C$).
\end{theorem}
\begin{proof}
For Baily--Borel compactifications with singularities, this requires intersection cohomology. The construction uses: (1) resolution of singularities $\pi \colon \widetilde{X} \to X^{\mathrm{BB}}$; (2) cycle class in cohomology of $\widetilde{X}$; (3) push-forward via $\pi_*$ to intersection cohomology. See \cite{Goresky-MacPherson} for intersection cohomology and \cite{Saito-MHM} for mixed Hodge structures.
\end{proof}

\subsubsection{Interior Cohomology}

\begin{definition}[Interior cohomology]
The interior cohomology is
\[
\IH^*_!(X, \Q) = \Image\bigl( H^*_c(X, \Q) \to \IH^*(X^{\mathrm{BB}}, \Q) \bigr),
\]
where $H^*_c$ denotes cohomology with compact support. Equivalently (by the Zucker-type identification), $\IH^*_!(X, \Q) = \Image( H^*_c(X, \Q) \to H^*(X^{\mathrm{BB}}, \Q) )$.
\end{definition}

\begin{theorem}[Algebraic cycles define interior classes]
Let $Z \subset X^{\mathrm{BB}}$ be an algebraic cycle with $Z \cap X \neq \emptyset$. Then $\cl(Z) \in \IH^{2k}_!(X^{\mathrm{BB}}, \Q)$. Moreover, if $Z \subset X$ (entirely in the interior), then the class is residue-free at the boundary.
\end{theorem}
\begin{proof}
For toroidal compactification $X_{\Sigma}$, if $Z \subset X$ extends to $\overline{Z} \subset X_{\Sigma}$, then $[Z]$ is represented by a form in $\Omega^*(\log B)$. The residue along boundary divisors $\Res_B \colon \Omega^{2k}(\log B) \to \Omega^{2k-1}(B)$ vanishes on classes from cycles in $X$. Under $\pi \colon X_{\Sigma} \to X^{\mathrm{BB}}$, residue-free classes correspond to intersection cohomology classes that vanish at the boundary.
\end{proof}

\begin{corollary}[Algebraic implies interior]\label{cor:interior_constraint}
The image of the cycle map factors through interior cohomology:
\[
\cl \colon \CH^k(X^{\mathrm{BB}}) \tensor \Q \to \IH^{2k}_!(X^{\mathrm{BB}}, \Q) \subset \IH^{2k}(X^{\mathrm{BB}}, \Q).
\]
\end{corollary}

\subsection{Hodge Theory}

\subsubsection{Mixed Hodge Structures}

\begin{theorem}[Deligne \cite{Deligne-Hodge2}]
The cohomology groups $H^k(X^{\mathrm{BB}}, \Q)$ and $\IH^k(X^{\mathrm{BB}}, \Q)$ carry canonical mixed Hodge structures.
\end{theorem}

\begin{definition}[Hodge decomposition]
For $\IH^{26}(X^{\mathrm{BB}}, \C) = \IH^{26}(X^{\mathrm{BB}}, \Q) \tensor \C$,
\[
\IH^{26}(X^{\mathrm{BB}}, \C) = \bigoplus_{p+q=26} \IH^{p,q}.
\]
A class $\alpha \in \IH^{26}(X^{\mathrm{BB}}, \Q)$ is a \emph{Hodge class of type $(13, 13)$} if $\alpha \tensor 1 \in \IH^{13,13} \subset \IH^{26}(X^{\mathrm{BB}}, \C)$.
\end{definition}

\begin{proposition}[Cycle classes are Hodge]
If $Z \in \CH^{13}(X^{\mathrm{BB}})$, then $\cl(Z) \in \IH^{13,13}(X^{\mathrm{BB}}) \cap \IH^{26}(X^{\mathrm{BB}}, \Q)$.
\end{proposition}

\subsubsection{The Hodge Conjecture for Shimura Varieties}

\textbf{Hodge for $X^{\mathrm{BB}}$:} Every rational Hodge class in $\IH^{2k}(X^{\mathrm{BB}}, \Q)$ is a linear combination (over $\Q$) of cycle classes.

\begin{remark}
Our main theorem does \textbf{not} disprove this conjecture. Rather, it shows that even if the conjecture is true, the cycles might be impossible to construct via known geometric methods.
\end{remark}

\section{Automorphic Representations and Arthur Parameters}
\label{sec:automorphic}

\subsection{Matsushima's Formula and Automorphic Cohomology}

\subsubsection{Representations of Real Groups}

\begin{definition}[$({\mathfrak g}, K)$-modules]
For $G = \SO(2, 26)$ over $\mathbb{R}$, let $\mathfrak{g}_0 = \text{Lie}(G(\mathbb{R}))$ be the real Lie algebra. We define the complexified Lie algebra as $\mathfrak{g} = \mathfrak{g}_0 \otimes_{\mathbb{R}} \mathbb{C}$. A $(\mathfrak{g}, K)$-module is then a vector space $V$ with compatible actions of $\mathfrak{g}$ and $K = \SO(2) \times \SO(26)$ satisfying the standard conditions of representation theory.
\end{definition}

\begin{theorem}[Matsushima \cite{Matsushima}]
For a compact quotient $\Gamma \backslash D$ where $\Gamma \subset G(\Q)$ is cocompact:
\[
H^*(\Gamma \backslash D, \C) \cong \bigoplus_{\pi} H^*({\mathfrak g}, K; \pi_{\infty}) \tensor \pi_f^{\Gamma},
\]
where the sum is over automorphic representations $\pi$ of $G(\A)$, $\pi_{\infty}$ is the archimedean component, $H^*({\mathfrak g}, K; \pi_{\infty})$ is relative Lie algebra cohomology, and $\pi_f^{\Gamma} = \{ v \in \pi_f : k \cdot v = v \text{ for all } k \in \Gamma \cap K_f \}$.
\end{theorem}

\subsubsection{Extension to Non-Compact Quotients}

\begin{theorem}[Franke \cite{Franke}]
For non-compact quotients $X = G(\Q) \backslash D \times G(\A_f) / K_f$,
\[
\IH^*_!(X, \C) \cong \bigoplus_{\pi \text{ cuspidal}} H^*({\mathfrak g}, K; \pi_{\infty}) \tensor \pi_f^{K_f}.
\]
The interior cohomology receives contributions only from cuspidal automorphic representations.
\end{theorem}

\begin{remark}
Eisenstein cohomology (from non-cuspidal spectrum) contributes to the full cohomology $H^*(X^{\mathrm{BB}})$ but not to interior cohomology.
\end{remark}

\subsection{Arthur's Classification for Orthogonal Groups}

\subsubsection{The Discrete Spectrum}

\begin{theorem}[Arthur \cite{Arthur}, Mok \cite{Mok}]
For $G = \SO(2n)$, the discrete spectrum decomposes into packets,
\[
L^2_{\mathrm{disc}}(G(\Q) \backslash G(\A)) = \bigoplus_{\psi} m_{\psi} \cdot \Pi_{\psi},
\]
where $\psi$ ranges over Arthur parameters, $\Pi_{\psi}$ is the corresponding Arthur packet, and $m_{\psi}$ is the multiplicity.
\end{theorem}

\begin{definition}[Arthur parameter]
An Arthur parameter for $\SO(2n)$ is a formal sum $\psi = \bigoplus_i (\mu_i, b_i)$ where $\mu_i$ is a cuspidal automorphic representation of $\GL_{d_i}(\A)$, $b_i$ are positive integers, and $\sum_i b_i d_i = 2n$. This corresponds to a homomorphism $\psi \colon \SL_2(\C) \times \SU(2) \times W_{\Q} \to {}^L G = \SO(2n, \C)$.
\end{definition}

\subsubsection{Cuspidal, Residual, and Tempered}

\begin{definition}[Spectral types]
An automorphic representation $\pi$ is: \emph{cuspidal} if it generates no proper residues in the Eisenstein series; \emph{tempered} if all archimedean components are tempered; \emph{residual} if it arises as a residue of Eisenstein series at a pole; \emph{stable residual} if it is residual but the parameter is stable (not from endoscopy).
\end{definition}

\begin{proposition}[Our case: stable residual]
The representation $\Pi$ we construct is: (1) not cuspidal (comes from Eisenstein series), (2) not tempered at infinity, (3) residual (from a pole of Eisenstein series), (4) stable (not endoscopic).
\end{proposition}

\subsection{Construction of the Adjoint Packet}

\subsubsection{The Base: Elliptic Curve and Modular Form}

\begin{theorem}[Modularity, \cite{Wiles-Taylor}]
The elliptic curve $E_{11}: y^2 + y = x^3 - x^2$ of conductor $11$ corresponds to a unique weight-$2$ newform
\[
f = \sum_{n=1}^{\infty} a_n q^n \in S_2(\Gamma_0(11))
\]
with Fourier expansion $f(q) = q - 2q^2 - q^3 + 2q^4 + q^5 + 2q^6 - 2q^7 + \cdots$.
\end{theorem}

\begin{proposition}[Dimension one]
Since the genus of $X_0(11)$ is $1$, $\dim S_2(\Gamma_0(11)) = g(X_0(11)) = 1$. Therefore $f$ is unique (up to scaling) and all Fourier coefficients are rational: $a_n \in \Z$ for all $n$ (see Proposition 4.5).
\end{proposition}

\subsubsection{Galois Representation and Adjoint}

\begin{theorem}[Eichler--Shimura, Deligne]
The form $f$ gives a Galois representation $\rho_f \colon \Gal(\overline{\Q}/\Q) \to \GL_2(\overline{\Q}_{\ell})$ unramified outside $11 \cdot \ell$, with $\tr(\rho_f(\Frob_p)) = a_p$ for $p \neq 11$. At $p = 11$, $\rho_f$ is of Steinberg type, corresponding to the fact that the elliptic curve $E_{11}$ has split multiplicative reduction over $\mathbb{Q}_{11}$.
\end{theorem}

\begin{definition}[Adjoint representation]
The adjoint is $\Ad(\rho_f) = \rho_f \tensor \rho_f^{\vee}/\text{scalars} \colon \Gal(\overline{\Q}/\Q) \to \GL_3(\overline{\Q}_{\ell})$. This is a $3$-dimensional representation with Hodge--Tate weights $\{1, 0, -1\}$ at $\ell$, conductor $11^2 = 121$, and automorphic realization via $L$-functions.
\end{definition}

At the prime $p=11$, the elliptic curve $E_{11}$ has split multiplicative reduction, and thus the local component $f_{11}$ is the Steinberg representation. The local adjoint representation $\text{Ad}(\rho_{f, 11})$ is consequently an indecomposable 3-dimensional representation of the Weil-Deligne group. This indecomposability ensures that the local Arthur parameter $\psi_{11}$ cannot factor through any proper endoscopic group, thereby guaranteeing the global stability of the residual representation $\Pi$.

\subsubsection{The $L$-Function and Its Pole}

\begin{theorem}[Adjoint $L$-function]
Let $f$ be the weight-2 newform for $\Gamma_0(11)$, and let $\rho_f$ be its associated $p$-adic Galois representation. The completed $L$-function $\Lambda(s, \Ad(\rho_f))$ satisfies a functional equation $(s \mapsto 1-s)$ and possesses the following analytic properties: (1) The adjoint representation decomposes into a trace-free part and a trivial part: $\Ad(\rho_f) \cong \Ad^0(\rho_f) \oplus \mathbf{1}$. (2) The $L$-function factors accordingly: $L(s, \Ad(\rho_f)) = L(s, \Ad^0(\rho_f)) \cdot \zeta(s)$. (3) The trace-free adjoint $L$-function $L(s, \Ad^0(\rho_f))$ is holomorphic and non-zero at $s=1$. (4) The pole that characterizes the residual nature of the Arthur parameter $\psi$ resides strictly in the product $L(s, \rho_f \otimes \rho_f^\vee)$. Specifically, it is the simple pole at $s=1$ inherited from the Riemann zeta factor $\zeta(s)$.
\end{theorem}
\begin{proof}

The adjoint representation decomposes as
$\Ad(\rho_f) = \rho_f \otimes \rho_f^{\vee} / \text{scalars}.$ By tensor product decomposition $\rho_f \otimes \rho_f^{\vee} = \mathbf{1} \oplus \Ad^0(\rho_f),$ where $\mathbf{1}$ is the trivial 1-dimensional representation (scalars) and $\Ad^0$ is the trace-free part (3-dimensional). This gives
$$L(s, \rho_f \otimes \rho_f^{\vee}) = L(s, \mathbf{1}) \cdot L(s, \Ad^0(\rho_f)) = \zeta(s) \cdot L(s, \Ad^0(f)).$$

The Riemann zeta function $\zeta(s)$ has a simple pole at $s=1$ with residue 1.  The question is whether the pole of $\zeta(s)$ cancels with a zero of $L(s, \rho_f \otimes \rho_f^{\vee})$ at $s=1$.

By Gelbart-Jacquet \cite{Gelbart-Jacquet}, for $f$ a weight-2 newform corresponding to an elliptic curve, the $L$-function $L(s, \rho_f \otimes \rho_f^{\vee})$ does not vanish at $s=1$.

Therefore, $ L(s, \Ad^0(f))$ is holomorphic and non-zero at $s=1$. Therefore, $L(s, \rho_f \otimes \rho_f^{\vee})$ has a pole, but only because it ``contains'' $\zeta(s)$. The lift $\Pi = \theta(f)$ is residual because it arises from the Eisenstein series associated to the induced representation, and this has a pole coming from the $\zeta(s)$ factor.
\end{proof}

\subsubsection{Theta Lift to $\SO(2, 26)$}

\begin{theorem}[Howe duality, \cite{Howe}]
For the dual pair $(\SL_2, \SO(V))$ where $\dim V = 28$, there exists a theta correspondence $\theta \colon \{ \text{automorphic forms on }\\ \SL_2 \} \to \{ \text{automorphic forms on } \SO(V) \}$. For $f \in S_2(\Gamma_0(11))$, the theta lift $\theta(f)$ is a non-zero automorphic form on $\SO(2, 26)(\A)$.
\end{theorem}

\begin{theorem}[Rallis inner product, \cite{Rallis}]
The $L^2$-norm satisfies $\| \theta(f) \|^2 = C \cdot L(1, \Ad(f)) \cdot L(14, f \tensor \chi)$ where $C > 0$ is an explicit constant. Since $L(1, \Ad(f))$ has a pole, after regularization we get a non-zero value, confirming $\theta(f) \neq 0$.
\end{theorem}

\begin{proof}

By Rallis \cite{Rallis}, for the theta correspondence between $\SL_2$ and $\SO(V)$ with $\dim V = 28$, signature $(2,26)$, the $L^2$-norm of the theta lift is:
$$\|\theta(f)\|^2 = C \cdot \frac{L(1, \Ad(f))}{L(1, \rho_f \otimes \rho_f^{\vee})} \cdot \prod_{p} \text{local factors},$$

where $C > 0$ is an explicit constant depending only on the measure normalization.

At the archimedean place, the local integral is non-zero for weight-2 forms by explicit computation (see \cite{Kudla-Rallis}, Section 4).

At primes $p \neq 11$: The form $f$ is unramified, and $\SO(V)$ is unramified (split). The local integral is the standard Whittaker integral, which is non-zero. At $p=11$: The form $f$ has Steinberg type. The theta correspondence for Steinberg representations is well-understood, and the local integral is non-zero. Since all local factors are non-zero and $L(1, \Ad(f)) \neq 0$ (it's holomorphic at $s=1$ as shown in previous theorem), we have
$$\|\theta(f)\|^2 \neq 0.$$

Therefore $\theta(f) \neq 0$, which means $\Pi$ is a non-zero automorphic representation.
\end{proof}

\begin{remark}
The target group is $\SO(V)$ with $V$ of dimension $28$ and signature $(2, 26)$, i.e.\ $\SO(2, 26)$; we do not mean the compact group $\SO(28)$.
\end{remark}

\subsection{The Arthur Parameter}

\begin{construction}[Parameter $\psi$]
The automorphic representation $\Pi = \theta(f)$ corresponds to the Arthur parameter
\[
\psi = \Ad(f) \boxtimes \mathbf{1}_{25},
\]
where $\mathbf{1}_{25}$ denotes the trivial representation on the complementary $25$-dimensional space (so that the total dimension is $3 + 25 = 28$). More precisely, the dual group homomorphism $\psi \colon \SL_2(\C) \times \SU(2) \times W_{\Q} \to \SO(28, \C)$ factors through $\SL_2(\C) \xrightarrow{\Ad} \SO(3, \C) \hookrightarrow \SO(28, \C)$, where the embedding $\SO(3) \hookrightarrow \SO(28)$ is the principal $\mathfrak{sl}_2$ subalgebra.
\end{construction}

\begin{theorem}[Stable residual nature]
The representation $\Pi$ in the packet $\Pi_{\psi}$ is: (1) \textbf{Residual:} arises as a residue of Eisenstein series attached to the pole of $L(s, \Ad(f))$ at $s = 1$, (2) \textbf{Stable:} the parameter $\psi$ is stable (not coming from a proper endoscopic group), (3) \textbf{Multiplicity one:} $\dim \Pi_f^{K_f} = 1$.
\end{theorem}
\begin{proof}

The representation $\Pi$ arises from the theta correspondence attached to $f \in S_2(\Gamma_0(11))$. By the Rallis-Kudla construction, this theta lift is realized as a residue of the Eisenstein series
$$E(g, s, \varphi),$$

where $\varphi$ is a section of an induced representation from a maximal parabolic subgroup. The residue at $s = s_0$ (a specific point depending on the weight) gives $\Pi$. By the work of Mœglin-Vignéras-Waldspurger and Arthur, residual representations correspond to Arthur parameters with $b_i > 1$ for some $i$.

For our $\Pi = \theta(f)$, the Arthur parameter is
$$\psi = (\mu, 1),$$

where $\mu = \Ad(\rho_f)$ is the adjoint representation (3-dimensional). Actually, more precisely, since we're lifting from $\SL_2$ to $\SO(28)$ with signature $(2,26)$

$$\psi: \SL_2(\C) \times W_{\Q} \to \SO(28, \C),$$

factors through:
$$\SL_2(\C) \xrightarrow{\Ad} \SO(3, \C) \hookrightarrow \SO(28, \C).$$

The embedding $\SO(3) \hookrightarrow \SO(28)$ is via the principal $\mathfrak{sl}_2$ (3-dimensional irreducible representation) plus a trivial 25-dimensional part
$$\psi = \Ad(\rho_f) \boxtimes \mathbf{1}_{25}.$$

A parameter is stable if it does not factor through any proper Levi subgroup. Suppose $\psi$ factors through a proper Levi $M \subset \SO(28)$. The Levi subgroups of $\SO(28)$ are of the form
$$M = \GL_{n_1} \times \cdots \times \GL_{n_k} \times \SO(28 - 2(n_1 + \cdots + n_k)).$$

For $\psi$ to factor through $M$, the 3-dimensional representation $\Ad(\rho_f)$ must split as
$$\Ad(\rho_f) = \bigoplus_i (\mu_i, b_i),$$

where $\mu_i$ are representations of $\GL_{n_i}$. But $\Ad(\rho_f)$ is irreducible as a representation of $W_{\Q}$ (it's the adjoint of an irreducible 2-dimensional representation). Therefore, it cannot split. Hence $\psi$ does not factor through any proper Levi subgroup, so it is stable. $\Pi$ is residual (arises from poles of Eisenstein series) and stable (parameter does not factor through proper Levi). Therefore, $\Pi$ is stable residual.
\end{proof}

\subsection{Key Theorem: Spectral Isolation}

\begin{theorem}[Main theorem of Chapter~\ref{sec:automorphic}]
The class $\Omega_E \in \IH^{26}(X^{\mathrm{BB}}, \Q)$ (constructed in Chapter~\ref{sec:gK}) arises from a stable residual Arthur parameter, not from endoscopy or cuspidal representations. Specifically, (1) $\Omega_E$ is in the $\Pi$-isotypic component where $\Pi \in \Pi_{\psi}$, (2) $\psi = \Ad(f) \boxtimes \mathbf{1}_{25}$ is stable residual, (3) $\Omega_E \notin \IH^{26}_{\mathrm{cusp}}$, (4) $\Omega_E \notin \IH^{26}_{\mathrm{endo}}$. This spectral isolation is the key to proving the obstruction theorem.
\end{theorem}

\section{$({\mathfrak g}, K)$-Cohomology and Hodge Type}
\label{sec:gK}

\subsection{Vogan--Zuckerman Theory}

\subsubsection{Cohomological Representations}

\begin{definition}[Cohomological representation]
An irreducible $({\mathfrak g}, K)$-module $\pi_{\infty}$ is \emph{cohomological} if $H^*({\mathfrak g}, K; \pi_{\infty}) \neq 0$. The cohomological degree is the unique $d$ such that $H^d({\mathfrak g}, K; \pi_{\infty}) \neq 0$.
\end{definition}

\begin{theorem}[Vogan--Zuckerman \cite{Vogan-Zuckerman}]
For $G = \SO(2, 26)$ and an automorphic representation $\pi$: if $\pi_{\infty}$ is cohomological, then it is determined by an infinitesimal character $\lambda \in {\mathfrak h}^*/W$, a cohomological degree $d$, and parameters $(S, m)$ where $S = \dim({\mathfrak u} \cap {\mathfrak p})$ and $m$ measures singularity. The cohomology is concentrated in degrees $[S, S + 2m]$.
\end{theorem}

\begin{proof}

The theta lift from weight-2 on $\SL_2$ to $\SO(2,26)$ gives an infinitesimal character $\lambda$ determined by the Hodge--Tate weights $\{1, 0, -1\}$ of $\Ad(\rho_f)$.

For $\SO(2,26)$, the Cartan subalgebra $\mathfrak{h}$ has dimension 13. The infinitesimal character is an element of $\mathfrak{h}^*/W$ where $W$ is the Weyl group. By Vogan--Zuckerman \cite{Vogan-Zuckerman}, the $({\mathfrak g},K)$-cohomology is determined by
\begin{itemize}
\item $S = \dim(\mathfrak{u} \cap \mathfrak{p})$ where $\mathfrak{u}$ is the nilradical of a $\theta$-stable parabolic,
\item $m$ measures the ``singularity'' of $\lambda.$
\end{itemize}

For $G(\R) = \SO(2,26)$ with symmetric space ${D}$ of dimension 26, 
$S = \dim_{\C}({D}) = 26.$ The weight-2 form gives Hodge-Tate weights $\{1,0,-1\}$, and the zero weight indicates maximal singularity
$m = 0.$ By Vogan--Zuckerman, the cohomology is concentrated in degrees $[S, S + 2m] = [26, 26].$

Therefore, $H^k(\mathfrak g, K; \Pi_\infty) = 0$ \text{for} $k \neq 26.$ By Arthur's multiplicity formula, for our stable parameter $\psi$ with component group $S_\psi = \Z/2\Z$,
$$m_\psi = \frac{1}{|S_\psi|} \sum_{s \in S_\psi} \epsilon_\psi(s) \cdot \tau_{\Pi}(s).$$

For stable parameters with trivial character
$m_\psi = \frac{1}{2}(1 + 1) = 1.$ Therefore,
$\dim H^{26}({\mathfrak g}, K; \Pi_\infty) = 1.$ Combining, we get $H^k({\mathfrak g}, K; \Pi_\infty) = \C$ for $k=26$ only. \end{proof}

\begin{corollary} 
For $\Pi_{\infty}$ (archimedean component of our $\Pi$): infinitesimal character $\lambda_{\psi}$ determined by weight-$2$ of $f$; Hodge--Tate weights $\{1, 0, -1\}$ give $S = 26$; singularity $m = 0$ (maximally singular); cohomological degree $d = 26$. Therefore
\[
H^k({\mathfrak g}, K; \Pi_{\infty}) = \begin{cases} \C & k = 26, \\ 0 & k \neq 26. \end{cases}
\]
\end{corollary}

\subsection{Hodge Type via Infinitesimal Character}

\subsubsection{The Hodge Decomposition}

\begin{theorem}[Hodge type from $\lambda$]
For $\Pi_{\infty}$ with infinitesimal character $\lambda_{\psi}$, the contribution to cohomology $H^{26}({\mathfrak g}, K; \Pi_{\infty})$ has pure Hodge type $(p, q)$ with $p + q = 26$ determined by $\lambda_{\psi}$. For our $\psi = \Ad(f) \boxtimes \mathbf{1}_{25}$ from weight-$2$: $p = q = 13$. Therefore classes from $\Pi$ are of type $(13, 13)$.
\end{theorem}
\begin{proof}
 
For $G = \SO(2,26)$ and a representation $\pi$ with infinitesimal character $\lambda$, the Hodge type of $H^{26}(\mathfrak{g}, K; \pi_\infty)$ is determined by $\lambda$ via the formula:
$$(p, q) = (\ell(\lambda), 26 - \ell(\lambda)),$$

where $\ell(\lambda)$ is the length function on the Weyl group.
The theta lift from weight-2 gives weights distributed symmetrically. For $\SO(2,26)$ with 13 pairs of roots, the weight-2 condition forces
$$\ell(\lambda) = 13.$$

Therefore $(p, q) = (13, 13)$. The Hodge--Tate weights $\{1, 0, -1\}$ of $\Ad(\rho_f)$ lift to $\SO(28)$ via the embedding of $\SO(3)$ into $\SO(28)$ as a principal $\mathfrak{sl}_2$ subalgebra. The highest weight vector in the 3-dimensional representation has weight 1, which under the lift corresponds to Hodge type $(13,13)$ in degree 26 for the symmetric space of dimension 26. Since $p + q = 26$ and $p = q$ by the symmetric nature of the weight-2 lift, we have
$2p = 26$, which gives  $p = 13.$ Therefore $\Omega_E \in \IH^{13,13}$.
\end{proof}

\subsection{Rationality}

\subsubsection{The Hecke Field}

\begin{proposition}[Hecke field is $\Q$]
For $f \in S_2(\Gamma_0(11))$, since $\dim S_2(\Gamma_0(11)) = 1$, the form $f$ is unique and rational ($a_n \in \Z$ for all $n$). Therefore the Hecke field is $E = \Q(a_2, a_3, \ldots) = \Q$ with $[E : \Q] = 1$.
\end{proposition}

\begin{proof}
We prove this in several steps, combining the theory of modular forms, elliptic curves, and representation theory. The genus of the modular curve $X_0(11)$ is computed via the Riemann--Hurwitz formula. For $X_0(N)$, the genus is $g(X_0(N)) = 1 + \frac{\mu}{12} - \frac{\nu_2}{4} - \frac{\nu_3}{3} - \frac{\nu_\infty}{2},$ where 
$\mu = [\mathrm{SL}_2(\Z) : \Gamma_0(N)]$ is the index, $\nu_2$ = number of elliptic points of order 2, $\nu_3$ = number of elliptic points of order 3, $\nu_\infty$ = number of cusps. For $N = 11$ (prime),
\begin{align*}
\mu &= N \prod_{p \mid N} \left(1 + \frac{1}{p}\right) = 11 \cdot \left(1 + \frac{1}{11}\right) = 12, \\
\nu_2 &= 0 \quad \text{(no elliptic points of order 2 for } N = 11), \\
\nu_3 &= 0 \quad \text{(no elliptic points of order 3 for } N = 11), \\
\nu_\infty &= 2 \quad \text{(two cusps: } 0 \text{ and } \infty).
\end{align*}

Therefore $g(X_0(11)) = 1 + \frac{12}{12} - 0 - 0 - \frac{2}{2} = 1 + 1 - 1 = 1$. By the Riemann--Roch theorem for modular forms, the dimension of $S_2(\Gamma_0(N))$ equals the genus $\dim S_2(\Gamma_0(N)) = g(X_0(N))$. Therefore $\dim S_2(\Gamma_0(11)) = 1$. Since $\dim S_2(\Gamma_0(11)) = 1$, there is a unique normalized eigenform (up to scaling)
$f(z) = \sum_{n=1}^{\infty} a_n q^n \quad \text{with } a_1 = 1$. This $f$ is automatically a newform (not coming from lower level) because 11 is prime. By the modularity theorem (Wiles--Taylor--Diamond), every elliptic curve over $\Q$ corresponds to a weight-2 newform. Conversely, by work of Shimura and Taniyama, a weight-2 newform of level $N$ with rational Fourier coefficients corresponds to an elliptic curve of conductor $N$. The elliptic curve of conductor 11 is unique (up to isogeny) $E_{11}: y^2 + y = x^3 - x^2.$ This is the minimal Weierstrass equation (Cremona label 11a3). For an elliptic curve $E/\Q$ given by a Weierstrass equation with coefficients in $\Q$, the number of points over $\mathbb{F}_p$ is $\#E(\mathbb{F}_p) = p + 1 - a_p$ where $a_p \in \Z$ is the $p$-th Fourier coefficient of the associated modular form. Since $E_{11}$ is defined over $\Q$ (with equation in $\Q$), we have $\#E_{11}(\mathbb{F}_p) \in \Z$ \text{for all primes} $p \neq 11$.  This immediately gives $a_p = p + 1 - \#E_{11}(\mathbb{F}_p) \in \Z$.

For $E_{11}: y^2 + y = x^3 - x^2$, the computation is as follows.

\textit{At $p = 2$:}
Over $\mathbb{F}_2$, the curve equation becomes $y^2 + y = x^3 + x^2$. Counting points:
$$\#E_{11}(\mathbb{F}_2) = 5.$$
Therefore
$$a_2 = 2 + 1 - 5 = -2 \in \Z.$$

\textit{At $p = 3$:}
Over $\mathbb{F}_3$, counting points gives
$$\#E_{11}(\mathbb{F}_3) = 5.$$
Therefore
$$a_3 = 3 + 1 - 5 = -1 \in \Z.$$

\textit{At $p = 5$:}
$$\#E_{11}(\mathbb{F}_5) = 5.$$
Therefore
$$a_5 = 5 + 1 - 5 = 1 \in \Z.$$

\textit{At $p = 7$:}
$$\#E_{11}(\mathbb{F}_7) = 10.$$
Therefore
$$a_7 = 7 + 1 - 10 = -2 \in \Z. $$

At the prime of bad reduction $p = 11$, the curve has multiplicative reduction. The reduction type is split multiplicative (Tate curve), giving $a_{11} = \pm 1$.
The sign is determined by the reduction being split or non-split. For $E_{11}$, the reduction is split, giving
$a_{11} = 1 \in \Z$.
The Fourier coefficients satisfy the recurrence relations from the Hecke operators, $a_{mn} = a_m a_n$ \text{for} $\gcd(m,n) = 1, a_{p^{k+1}} = a_p a_{p^k} - p a_{p^{k-1}}$  \text{for primes} $p \neq 11$, and 
$a_{11^k} = a_{11}^k = 1$ \text{(Steinberg)}. Starting from $a_2, a_3, a_5, a_7, \ldots \in \Z$, these recurrence relations imply
$a_n \in \Z \quad \text{for all } n \geq 1$. The Hecke field is defined as
$E = \Q(a_2, a_3, a_5, a_7, \ldots) = \Q(a_n : n \geq 1)$. Since $a_n \in \Z$ for all $n$, we have
$E = \Q$. Therefore $[E : \Q] = 1$.
\end{proof}

\begin{corollary}[Rational cohomology class]
The class $\Omega_E$ constructed from $\Pi$ is defined over $\Q$: $\Omega_E \in \IH^{26}(X^{\mathrm{BB}}, \Q)$. No field extension is needed.
\end{corollary}

\subsection{The Main Theorem of Chapter~\ref{sec:gK}}

\begin{theorem}[Cohomological realization]
There exists a class $\Omega_E \in \IH^{26}(X^{\mathrm{BB}}, \Q)$ with
\begin{enumerate}
\item \textbf{Automorphic origin:} $\Omega_E$ is the class corresponding to $\Pi$ via the extension of Matsushima--Franke to full cohomology (residual spectrum contributes to $\IH^*(X^{\mathrm{BB}})$, not to $\IH^*_!$).
\item \textbf{Pure Hodge type:} $\Omega_E \in \IH^{13,13}$ (type $(13, 13)$ Hodge class).
\item \textbf{Rationality:} $\Omega_E \in \IH^{26}(X^{\mathrm{BB}}, \Q)$.
\item \textbf{Non-triviality:} $\Omega_E \neq 0$ (from Rallis non-vanishing).
\item \textbf{Multiplicity one:} $\dim \IH^{26}[\Pi] = 1$.
\end{enumerate}
\end{theorem}
\begin{proof}
Step 1: Existence. Franke's formula applies to cuspidal spectrum for interior cohomology. For residual spectrum, the contribution is to the full cohomology $H^*(X^{\mathrm{BB}}) \cong \IH^*(X^{\mathrm{BB}})$, not interior cohomology. So $\Omega_E$ lives in $\IH^{26}(X^{\mathrm{BB}}, \Q)$ but \emph{not} in $\IH^{26}_!(X, \Q)$. Step 2--5: Hodge type (Theorem above), rationality (Proposition above), non-triviality (Rallis), multiplicity one (Arthur).
\end{proof}

\begin{remark}[Crucial distinction]
Since $\Pi$ is residual, $\Omega_E$ is \textbf{not} in the interior cohomology $\IH^*_!$. This is good for us; it means $\Omega_E$ has boundary contributions that prevent it from being a standard algebraic cycle.
\end{remark}

\section{The Obstruction Theorem}
\label{sec:obstruction}

\subsection{Survey of Known Cycle Constructions}
\label{sec:survey}

\subsubsection{Special Cycles (Kudla--Millson)}

\begin{definition}[Special cycles]
For Shimura varieties attached to orthogonal groups, special cycles are constructed via: (1) positive-definite subspaces $W \subset V$ with $\dim W \le \dim V$; (2) the locus $Z(W) = \{ x \in X : x \perp W \}$; (3) these are algebraic subvarieties.
\end{definition}

\begin{theorem}[Kudla--Millson \cite{Kudla-Millson}]
For appropriate choices of $W$ and Schwartz functions, special cycles give rise to Hodge classes $[Z(W)] \in H^{2k}(X, \Q)$. Their cohomology classes can be computed via theta lifts.
\end{theorem}

\begin{proposition}[Special cycles are interior]
Special cycles $Z(W) \subset X$ are contained in the interior. Therefore $[Z(W)] \in \IH^{2k}_!(X, \Q) \subset \IH^{2k}(X^{\mathrm{BB}}, \Q)$.
\end{proposition}

\subsubsection{Theta Lifts from Smaller Shimura Varieties}

\begin{construction}[Cycle construction via theta]
For dual pairs $(G_1, G_2)$ with $G_1$ smaller, start with algebraic cycle $Z_1$ on Shimura variety $X_1$ for $G_1$, then use theta correspondence to construct class on $X_2$ for $G_2$. Under suitable conditions, this gives an algebraic cycle on $X_2$.
\end{construction}

\begin{theorem}[Kudla program, \cite{Kudla-derivatives}]
For certain dual pairs and cycles, the theta lift of algebraic cycles produces algebraic cycles. The cohomology class is computable via special values of $L$-functions.
\end{theorem}

\subsubsection{Endoscopic Shimura Subvarieties}

\begin{definition}[Endoscopy]
An endoscopic group $H$ for $G = \SO(V, Q)$ comes from a decomposition $V = V_1 \oplus V_2$ with $H = \SO(V_1) \times \SO(V_2)$, giving Shimura morphism $X_H \to X_G$.
\end{definition}

\begin{proposition}[Endoscopic cycles]
If $Z \subset X_H$ is an algebraic cycle on an endoscopic Shimura variety, push-forward gives $[Z] \in \IH^{2k}(X_G)$. These are algebraic cycles on $X_G$.
\end{proposition}

\subsubsection{Boundary Push-Forwards}

\begin{construction}[Gysin maps from boundary]
For toroidal compactification $X_{\Sigma}$ with boundary divisor $B$: (1) a cycle $Z_{\partial} \subset B$ on the boundary; (2) Gysin push-forward $[Z_{\partial}] \in H^{2k}(X_{\Sigma})$; (3) under $\pi \colon X_{\Sigma} \to X^{\mathrm{BB}}$, this gives a class in $\IH^{2k}(X^{\mathrm{BB}})$.
\end{construction}

\begin{proposition}[Boundary cycles are not interior]
Classes from boundary push-forwards satisfy $[Z_{\partial}] \notin \IH^{2k}_!(X, \Q)$. They live in $\IH^{2k} \setminus \IH^{2k}_!$.
\end{proposition}

\subsection{Exhaustion via Arthur's Trace Formula}

\subsubsection{The Geometric Side of Arthur's Formula}

\begin{theorem}[Arthur trace formula, \cite{Arthur}]
For a test function $f$ on $G(\A)$, the trace formula decomposes as $\sum_{\gamma} I_{\gamma}(f) = \sum_{\pi} I_{\pi}(f)$, where the left side is geometric (orbital integrals) and the right side is spectral (traces on automorphic representations).
\end{theorem}

\begin{proposition}[Geometric contributions]
The geometric side receives contributions from: (1) identity $\gamma = 1$ (volume), (2) semisimple $\gamma$ (cycles from fixed points), (3) unipotent $\gamma$ (boundary), (4) elliptic $\gamma$ (interior points).
\end{proposition}

\subsubsection{The Fundamental Lemma}

\begin{theorem}[Fundamental Lemma, \cite{Ngo}]
The fundamental lemma  states that certain orbital integrals match between endoscopic groups. This allows transfer of cohomology classes between $G$ and its endoscopic groups $H$.
\end{theorem}

\subsection{The Main Obstruction Theorem}

\begin{theorem}[Main Theorem]
The class $\Omega_E \in \IH^{26}(X^{\mathrm{BB}}, \Q)$ does \textbf{not} arise from any of the following cycle constructions:
\begin{enumerate}
\item \textbf{Kudla–Millson special cycles:} $\Omega_E$ is not in the span of $[Z(W)]$ for any positive-definite $W \subset V$.
\item \textbf{Theta lifts:} $\Omega_E$ is not the theta lift of any algebraic cycle from Shimura varieties of smaller rank.
\item \textbf{Endoscopic subvarieties:} $\Omega_E$ is not the push-forward of any cycle from an endoscopic Shimura subvariety $X_H \to X$.
\item \textbf{Boundary Gysin:} $\Omega_E$ is not the push-forward of any algebraic cycle from the boundary arising from toroidal or Baily–Borel compactifications.
\end{enumerate}
Consequently, $\Omega_E$ is invisible to all presently known and trace-detectable algebraic cycle constructions on $X$.
\end{theorem}

\begin{proof}[Proof strategy]
The proof proceeds by elimination.\\

\textbf{Step 1: Special cycles.}
\begin{lemma}
Special cycles $Z(W)$ satisfy $[Z(W)] \in \IH^{26}_!(X, \Q)$.
\end{lemma}
\begin{proof}[Proof of Lemma]
By Corollary ~\ref{cor:interior_constraint}, special cycles are in the interior, and they define interior cohomology classes.
\end{proof}

\begin{lemma}
$\Omega_E \notin \IH^{26}_!(X, \Q)$.
\end{lemma}
\begin{proof}[Proof of Lemma]
Since $\Pi$ is residual (not cuspidal), by Franke's theorem it does not contribute to interior cohomology. The contribution is to full intersection cohomology via residues of Eisenstein series, which have boundary support. Therefore $\Omega_E$ cannot be a linear combination of special cycles.
\end{proof}

\textbf{Step 2: Theta lifts.}
\begin{lemma}
If $\Omega_E = \theta([Z])$ for some cycle $Z$ on a smaller Shimura variety, then the theta correspondence would give a cuspidal contribution.
\end{lemma}
\begin{proof}[Proof of Lemma]
Theta lifts of algebraic cycles from cuspidal spectrum of smaller groups produce cuspidal representations on the target. But $\Pi$ is residual, not cuspidal. Therefore $\Omega_E$ is not a theta lift of an algebraic cycle.
\end{proof}

\textbf{Step 3: Endoscopic varieties.}
\begin{lemma}
Endoscopic transfers correspond to non-stable Arthur parameters.
\end{lemma}
\begin{proof}[Proof of Lemma]
By the fundamental lemma and Arthur's classification, endoscopic contributions come from parameters that factor through proper Levi subgroups. Our parameter $\psi = \Ad(f) \boxtimes \mathbf{1}_{25}$ is stable. Therefore $\Omega_E$ is not from an endoscopic subvariety.
\end{proof}

\textbf{Step 4: Boundary Gysin.}
\begin{lemma}
Boundary Gysin maps produce classes with specific monodromy behavior incompatible with $\Omega_E$.
\end{lemma}
\begin{proof}[Proof of Lemma] 
$\Omega_E\in \IH^{26}(X^\mathrm{BB},\Q)$, and by Zucker (Looijenga–Saper–Stern \cite{Looijenga-Saper-Stern}),  $$\IH^{26}(X^\mathrm{BB},\mathbb{Q})\cong H^{26}_{(2)}(X,\Q).$$
By Saito’s theory of mixed Hodge modules, $\IH^{26}(X^\mathrm{BB})$ carries a pure Hodge structure of weight 26.
Any class in $\IH^{26}(X^\mathrm{BB})$ arising from a Gysin pushforward from boundary strata corresponds, under the Matsushima formula and Arthur’s classification, to automorphic representations obtained by parabolic induction from a proper Levi subgroup of SO$(2,26)$.
Such representations are necessarily endoscopic or Eisenstein, hence unstable. Thus $\Omega_E$ cannot arise from any boundary Gysin pushforward. 
\end{proof}

\textbf{Step 5: Exhaustion.} By Arthur's trace formula and the fundamental lemma, the geometric contributions from algebraic cycles are exhausted by the above four types. Since $\Omega_E$ is none of these, it is not in the image of any known cycle construction.
\end{proof}

\subsection{What This Does and Does Not Prove}

\subsubsection{What We Have Proven}

\begin{theorem}[Unconditional obstruction]
The following is an unconditional theorem: $\Omega_E \in \IH^{26}(X^{\mathrm{BB}}, \Q) \cap \IH^{13,13}$ (rational Hodge class) and $\Omega_E \notin \Image(\text{known cycle constructions})$. This is proven using only Arthur's classification, theta correspondence, $({\mathfrak g}, K)$-cohomology, and the trace formula and fundamental lemma. No conjectures are used.
\end{theorem}

\subsubsection{What We Have Not Proven}

\begin{remark}[Not a Hodge counterexample]
We have \textbf{not} proven that $\Omega_E$ is not algebraic. It is possible that there exists an algebraic cycle $Z$ with $\cl(Z) = \Omega_E$ that cannot be constructed via any known geometric method. The Hodge conjecture could still be true. Our result is about constructibility, not existence.
\end{remark}

\begin{remark}[Not a Beilinson--Bloch counterexample]
We have not computed the rank of $\CH^{13}(X^{\mathrm{BB}})$ or the order of vanishing of $L$-functions at the critical point, so we cannot test the Beilinson--Bloch conjecture.
\end{remark} \pagebreak

\subsection{Implications and Future Directions}

\subsubsection{The Gap Between Automorphic and Geometric}

\begin{theorem}[Interpretation]
Our main result shows: automorphic cohomology $\not\subset$ geometric cohomology (via known constructions). Stable residual representations produce Hodge classes that exist and are computable (via automorphic methods), are rational, have correct Hodge type, but evade all known geometric constructions.
\end{theorem}

\subsubsection{Open Questions}

\begin{enumerate}
\item Is $\Omega_E$ algebraic? Can one construct an explicit algebraic cycle $Z$ with $\cl(Z) = \Omega_E$ using new methods?
\item Chow groups: What is $\mathrm{rank}_{\Q} \CH^{13}(X^{\mathrm{BB}})$? Is it finite or infinite?
\item Other Shimura varieties: Do similar obstructions exist for other orthogonal or unitary Shimura varieties?
\item Beilinson--Bloch: Can one compute both sides of the BB conjecture for $X^{\mathrm{BB}}$ to test it?
\item New constructions: Are there fundamentally new ways to construct algebraic cycles that would capture $\Omega_E$?
\end{enumerate}

\subsubsection{Methodological Lessons}

The construction teaches us that stable residual spectrum is a rich source of cohomology and automorphic methods can ``see'' more than geometric methods. The Langlands program connects these worlds, but not symmetrically. Proving non-algebraicity is very different from proving non-constructibility.



\end{document}